\documentclass[10pt]{amsart}

\newsavebox{\smlmat}% Box to store smallmatrix content
\savebox{\smlmat}{$\left(\begin{smallmatrix}\sqrt{2/\sqrt{3}}&\frac{1}{\sqrt{2\sqrt{3}}} \\ 0 & \sqrt{\sqrt{3}/2} \end{smallmatrix}\right)$}
\newsavebox{\ssmlmat}% Box to store smallmatrix contents
\savebox{\ssmlmat}{$\left(\begin{smallmatrix}\sqrt{2/\sqrt{3}}&-\frac{1}{\sqrt{2\sqrt{3}}} \\ 0 & \sqrt{\sqrt{3}/2} \end{smallmatrix}\right)$}
\newsavebox{\sssmlmat}% Box to store smallmatrix content
\savebox{\sssmlmat}{$\left(\begin{smallmatrix} e^{u} &0\\ 0 & e^{-u}\end{smallmatrix} \right)
 \left(\begin{smallmatrix} 1 & s \\ 0 &1\end{smallmatrix} \right)
$}

\usepackage{times,amsfonts,amsmath,amstext,amsbsy,amssymb,
  amsopn,amsthm,upref,eucal, amscd}
\usepackage[colorlinks=true]{hyperref}
\usepackage[T1]{fontenc}
\usepackage{hyperref}
\usepackage[backgroundcolor=white]{todonotes}
\usepackage{multirow}
\usepackage{color}
\usepackage{comment}
\usepackage{graphicx}

\newtheorem{theorem}{Theorem}[section]
\newtheorem{lemma}[theorem]{Lemma}
\newtheorem{corollary}[theorem]{Corollary}
\newtheorem{proposition}[theorem]{Proposition}

\numberwithin{equation}{section}
\newtheorem{conjecture}[theorem]{Conjecture}

\theoremstyle{definition}
\newtheorem{definition}[theorem]{Definition}
\newtheorem{example}[theorem]{Example}

\newtheorem{remark}[theorem]{Remark}

\def\leq{\leqslant }
\def\geq{\geqslant}

\usepackage{accents}

\begin{document}

\title{Minimizing entropy for translation surfaces}

\author{Paul Colognese and Mark Pollicott}

\date{}
\address{P. Colognese, Department of Mathematics, Warwick University, Coventry,
   CV4 7AL, UK.}
\email{paul.colognese@gmail.com}

\address{M. Pollicott, Department of Mathematics, Warwick University, Coventry,
   CV4 7AL, UK.}
\email{masdbl@warwick.ac.uk}

\thanks{The authors would like to thank L.B\'etermin for suggesting the connection with his work \cite{betermin} and Samuel Leli\`evre for helpful comments.   We are also grateful to the referee for their helpful comments and suggestions. The second author was supported by ERC Grant
833802-Resonances  and EPSRC grant EP/T001674/1.}

\begin{abstract}
In this note we consider  the entropy  \cite{{dankwart}} of  unit area translation surfaces in the $SL(2, \mathbb R)$ orbits of square tiled surfaces that are the  union of  squares,  
where the singularities occur at the vertices and  the singularities have a common cone angle.
We show that the entropy over such orbits   is minimized 
at those surfaces
 tiled by equilateral triangles where the singularities occur precisely at the vertices. We also provide a method for approximating the entropy of surfaces in the orbits.
\end{abstract}
\maketitle

\section{Introduction}

We begin by  recalling  for the purposes of motivation  a well known classical  result of Katok from 1982 for compact  negatively curved surfaces. 
 Let $\mathcal M_g$ denote the space of negatively curved $C^\infty$ Riemannian metrics of unit volume on a compact orientable surface of genus $g \geq 2$.   
 The
 % (volume) 
 entropy function $h: \mathcal M_g \to \mathbb R^+$ 
 can be defined in terms of 
 %free homotopy classes and
the  growth rate of 
    closed geodesics 
 $$
 h(\rho) = \lim_{T \to +\infty} \frac{1}{T} \log \#\{\gamma \hbox{ : } \ell_\rho(\gamma) \leq T\}
 $$
 where
 $\rho \in \mathcal M_g$ and
  $\ell_\rho(\gamma)$ denotes  the length of a closed $\rho$-geodesic $\gamma$.
  % corresponding to a free homotopy class $\gamma$.
 %was defined in \cite{manning}
%This is known to be a $C^\infty$  \cite{KKPW}.\footnote{Here $\mathcal M_g$ has the structure of a Banach manifold}
When restricted to metrics  of unit  volume the entropy 
%$h(\cdot)$ 
is minimized  precisely at metrics of constant curvature \cite{katok}.

In this note we want to formulate a partial analogue of this result for translation surfaces. Informally, a translation surface can be thought of as a closed surface obtained from taking a collection of polygons in the plane and gluing together parallel edges via isometries (see \cite{zorich} for a good introduction to translation surfaces). A translation surface has a finite number of singularities with cone angles of the form $2\pi(k+1)$ where $k\in\mathbb N$. To see what this means, consider the following construction: let $k\in\mathbb N$ and take $(k+1)$ copies of the upper half-plane with the usual metric and $(k+1)$ copies of the lower half-plane. Then glue them together along the half infinite rays $[0,\infty)$ and $(-\infty,0]$ in cyclic order (Figure \ref{fig:nbhdsing}). 

\begin{figure}[h!]
    \centering
    \begin{tikzpicture}[scale=0.7]
      \draw (0,0)--(3,0);
      \draw (4,0)--(7,0);
      \draw (8,0)--(11,0);
      \draw (12,0)--(15,0);
      
      \draw[dashed] (0,0) arc (180:360:1.5);
        \draw[dashed] (7,0) arc (0:180:1.5);
              \draw[dashed] (8,0) arc (180:360:1.5);
                    \draw[dashed] (15,0) arc (0:180:1.5);
                    
      \draw [fill= red] (1.5,0) circle (0.1);
            \draw [fill= red] (5.5,0) circle (0.1);
                  \draw [fill= red] (9.5,0) circle (0.1);
                        \draw [fill= red] (13.5,0) circle (0.1);

          \node[above] at (0.75,0) {$A$};
                    \node[above] at (2.25,0) {$B$};
                \node[above] at (8.75,0) {$C$};
                      \node[above] at (10.25,0) {$D$};
                  \node[below] at (4.75,0) {$C$};
              \node[below] at (6.25,0) {$B$};
              \node[below] at (12.75,0) {$A$};
             \node[below] at (14.25,0) {$D$};

    \end{tikzpicture}
    \caption{Four half-disks glued together cyclically. A singularity/cone-point of angle $4\pi$ on a translation surface has a neighborhood isometric to a neighborhood of the the origin in the picture (the red dot). }
    \label{fig:nbhdsing}
\end{figure}
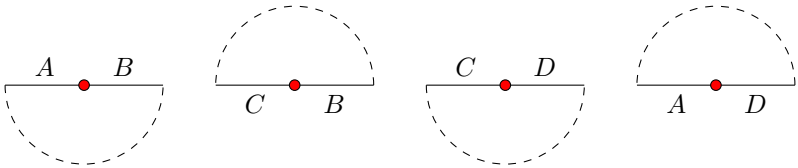

There are a few equivalent definitions of translation surfaces that appear in the literature. We will use the following definition (see \cite{wright}).

\begin{definition}\label{def:transsurface}
A translation surface is a closed and connected topological surface, $X$, together with a finite set of points $\Sigma$ and an atlas of charts to $\mathbb C$ on $X\backslash \Sigma$, whose transition maps are translations. Furthermore, we require that for each point $x\in \Sigma$, there exists some $k\in\mathbb N$ and a homeomorphism of a neighborhood of $x$ to a neighborhood of the origin in the $2k+2$ half-plane construction that is an isometry away from $x$. 
\end{definition}

Dankwart  associated to a translation surface $X$ with a non-empty finite singularity set  $\Sigma$ an analogous notion of  entropy 
%$h(X)>0$
 \cite{dankwart}.   
Given $k_1, \ldots, k_n \geq 1$ we can denote by $\mathcal H^1(k_1, \ldots, k_n)$ the space of
unit area  translation 
surfaces with  $n$ singularities in $\Sigma$ with cone angles $2\pi (k_1+1), \ldots, 2\pi (k_n+1)$.
%containing $X$. 
%$g \geq 0$ and $\underline k \in \mathbb N^m$  the strata $\mathcal M(g, \underline k)$ consists of the translation surfaces of genus $g$ and have $m$ singularities corresponding to the vector $\underline k$. \footnote{Is this true? or $C^\infty$}
%Moreover, he showed that the function 
The entropy function  $h: \mathcal H^1(k_1, \ldots, k_n) \to \mathbb R^+$
can be defined in terms of  
%free homotopy classes and
 the growth rate of closed  geodesics  containing  a singular point 
 $$
 h(X) = \lim_{T \to +\infty} \frac{1}{T} \log \#\{\gamma \hbox{ : } \gamma \cap \Sigma \neq \emptyset,  \ell_X(\gamma) \leq T\}
 $$
 where $\ell_X(\gamma)$ denotes  the length of a closed geodesic $\gamma$
 on $X$ which includes a singular point from $\Sigma$.\footnote{This avoids the complication of  accounting for cylinders of uncountably many parallel geodesics. 
 Alternatively, we could account for these by counting only their free homotopy classes, but then their polynomial growth does not affect the definition of the entropy.}
 The entropy is continuous and bounded below, and can become arbitrarily large (when a closed geodesic becomes sufficiently small) even when the total area is normalized (although it will always be finite) \cite{dankwart}.
 %The entropy 
%$h: \mathcal M_X \to \mathbb R^+$ 
%continuous.\footnote{perhaps smooth?}
%We need to impose some additional conditions.
We  restrict 
the type of  translation surfaces we will consider as follows:
\begin{enumerate}
\item
Firstly, 
we  fix  a unit area  square tiled surface
$X_0$, 
which is a union of squares 
where the singularities occur at the vertices and  the singularities have a common cone angle; and 
%is triangulated by equilateral triangles
%\footnote{Of the same size} for which the vertices are singularities with the same cone angle;
%we assume that  the singularities have the same cone angle, $k \geq 1$, say,
%$\mathcal M(g, \underline k)$
\item
Secondly,  we consider the  three dimensional orbit $SL(2, \mathbb  R) X_0$ associated to the linear action of the group
 $SL(2, \mathbb  R)$ \cite{wright}.
\end{enumerate}
%We say that an equilateral translation surface $X_0$ is one which is triangulated by equilateral triangles
% whose vertices are cone singularities.   

The surfaces described in (1) are known as square tiled surfaces (see \cite{matheus} for a good introduction to square tiled surfaces). 

Note that the area of surfaces in the orbit $SL(2, \mathbb R) X_0$ coincide with the area of $X_0$.

Our main result is the following theorem.

\begin{theorem}\label{main}
%If the  singularities $\Sigma \subset X_0$ have the same cone angle,
If $X_0$ satisfies the hypotheses in (1) 
 then the entropy function
$$h: SL(2, \mathbb  R) X_0 \to \mathbb R^+$$ 
is minimized at equilateral translation surfaces, by which we mean translation surfaces tiled by equilateral triangles where the singularities occur precisely at the vertices of the triangles.
\end{theorem}

We take the convention that we  identify  surfaces that  are  identical up to a rotation (i.e., the action of $SO(2)$).

Theorem \ref{main} applies  to the following simple example and to the examples listed in \S3.
Furthermore, it is known that every stratum contains an equilateral translation  surface
 \cite{BG}.

\begin{example}\label{lshape}
Let $X_0$ be the  $L$-shaped square tiled translation surface made up of three squares  
(see Figure 1).   The surface $X_0$ has genus $2$ and a single singularity of cone angle $6\pi$.
The $SL(2,\mathbb R)$-orbit,  $SL(2, \mathbb  R)X_0$,  contains one
equilateral translation surface up to isometry.  
This  surface  globally minimizes entropy 
in the orbit space.
%\footnote{A slight subtlety is that some authors triple count such surfaces by allowing 
%surfaces to differ under rotation.}
\end{example}

\begin{figure}[h!]
\begin{tikzpicture}[scale=0.4, every node/.style={scale=0.6}]
\node at (-2,5) {\hbox{\huge (a)}};
\draw[gray, thick] (0,0)--(6,0);
\draw[gray, thick] (3,0)--(3,6);
\draw[gray, thick] (0,0)--(0,6);
\draw[gray, thick] (3,0)--(3,3);
\draw[gray, thick] (6,0)--(6,3);
\draw[gray, thick] (3,6)--(0,6);
%\draw[gray, thick] (6,0)--(7,0);
%\draw[gray, thick, dashed] (7,0)--(8,0);
\draw[gray, thick] (0,3)--(6,3);
\draw[gray, thick] (0,0)--(0,3);
\draw[gray, thick] (3,0)--(3,3);
\draw[gray, thick] (6,0)--(6,3);
\draw[gray, thick] (6,9)--(6,9);
\draw [fill=red] (0,0) circle [radius=0.15];
\draw [fill=red] (0,3) circle [radius=0.15];
\draw [fill=red] (3,3) circle [radius=0.15];
\draw [fill=red] (3,6) circle [radius=0.15];
\draw [fill=red] (0,6) circle [radius=0.15];
%\draw [fill=red] (6,9) circle [radius=0.15];
\draw [fill=red] (3,0) circle [radius=0.15];
\draw [fill=red] (6,0) circle [radius=0.15];
\draw [fill=red] (6,3) circle [radius=0.15];
\node at (1.5,-0.4) {$1$};
\node at (4.5,-0.4) {$2$};
\node at (1.5,6.4) {$1$};
\node at (4.5,3.4) {$2$};
\node at (-0.4,1.5) {$a$};
\node at (6.4,1.5) {$a$};
\node at (-0.4,4.5) {$b$};
\node at (3.4, 4.5) {$b$};
\end{tikzpicture}
\hskip 0.25cm
\begin{tikzpicture}[scale=0.45, every node/.style={scale=0.6}]
\node at (0,4.5) {\hbox{\huge (b)}};
\draw[gray, thick] (0,0)--(6,0);
\draw[gray, thick] (0,0)--(1.5,2.6);
\draw[gray, thick] (1.5,2.6)--(3,5.2);
%\draw[gray, thick] (3,0)--(3,0);
%\draw[gray, thick] (3,0)--(3,6);
%\draw[gray, thick] (6,0)--(6,6);
%\draw[gray, thick] (3,6)--(6,6);
%\draw[gray, thick] (6,0)--(7,0);
%\draw[gray, thick, dashed] (7,0)--(8,0);
%\draw[gray, thick] (1.5,3)--(7.5,3);
%\draw[gray, thick] (0,0)--(0,3);
%\draw[gray, thick] (3,0)--(3,3);
%\draw[gray, thick] (6,0)--(6,3);
\draw[gray, thick] (1.5,2.6)--(7.5,2.6);
%\draw[gray, thick] (6,7.8)--(6,7.8);
\draw[gray, thick] (3,5.2)--(6,5.2);
\draw[gray, thick] (6,0)--(7.5,2.6);
\draw[gray, thick] (3,0)--(6,5.2);
\draw[gray, dashed] (3,0)--(1.5,2.6);
\draw[gray, dashed] (6,0)--(4.5,2.6);
\draw[gray, dashed] (3,5.2)--(4.5, 2.6);
\draw [fill=red] (0,0) circle [radius=0.15];
\draw [fill=red] (1.5,2.6) circle [radius=0.15];
\draw [fill=red] (4.5,2.6) circle [radius=0.15];
\draw [fill=red] (6,5.2) circle [radius=0.15];
\draw [fill=red] (3,5.2) circle [radius=0.15];
%\draw [fill=red] (6,9) circle [radius=0.15];
\draw [fill=red] (3,0) circle [radius=0.15];
\draw [fill=red] (6,0) circle [radius=0.15];
\draw [fill=red] (7.5,2.6) circle [radius=0.15];
\node at (1.5,-0.4) {$1$};
\node at (4.5,-0.4) {$2$};
\node at (6,3.2) {$2$};
\node at (4.5,5.8) {$1$};
\node at (0,1.3) {$a$};
\node at (7.4,1.3) {$a$};
\node at (1.7,4.0) {$b$};
\node at (5.7, 4.0) {$b$};
\end{tikzpicture}
%\hskip 0.25cm
\begin{tikzpicture}[scale=0.45, every node/.style={scale=0.6}]
\node at (3.5,4.5) {\hbox{\huge (c)}};
\draw[gray, thick] (0,0)--(6,0);
\draw[gray, dashed] (0,0)--(1.5,2.6);
\draw[gray, dashed] (-1.5,2.6)--(0,5.2);
\draw[gray, thick] (1.5,2.6)--(0,5.2);
\draw[gray, thick] (-1.5,2.6)--(-3,5.2);
%\draw[gray, thick] (3,0)--(3,0);
%\draw[gray, thick] (3,0)--(3,6);
%\draw[gray, thick] (6,0)--(6,6);
%\draw[gray, thick] (3,6)--(6,6);
%\draw[gray, thick] (6,0)--(7,0);
%\draw[gray, thick, dashed] (7,0)--(8,0);
%\draw[gray, thick] (1.5,3)--(7.5,3);
%\draw[gray, thick] (0,0)--(0,3);
%\draw[gray, thick] (3,0)--(3,3);
%\draw[gray, thick] (6,0)--(6,3);
\draw[gray, thick] (-1.5,2.6)--(4.5,2.6);
%\draw[gray, thick] (6,7.8)--(6,7.8);
\draw[gray, thick] (-3,5.2)--(0,5.2);
\draw[gray, thick] (6,0)--(4.5,2.6);
%\draw[gray, thick] (3,0)--(6,5.2);
\draw[gray, dashed] (3,0)--(4.5,2.6);
\draw[gray, thick] (3,0)--(1.5,2.6);
\draw[gray, dashed] (6,0)--(4.5,2.6);
%\draw[gray, dashed] (6,5.2)--(7.5, 2.6);
\draw [fill=red] (0,0) circle [radius=0.15];
\draw [fill=red] (-1.5,2.6) circle [radius=0.15];
\draw [fill=red] (1.5,2.6) circle [radius=0.15];
\draw [fill=red] (-3,5.2) circle [radius=0.15];
\draw [fill=red] (0,5.2) circle [radius=0.15];
%\draw [fill=red] (6,9) circle [radius=0.15];
\draw [fill=red] (3,0) circle [radius=0.15];
\draw [fill=red] (6,0) circle [radius=0.15];
\draw [fill=red] (4.5,2.6) circle [radius=0.15];
\node at (1.5,-0.4) {$1$};
\node at (4.5,-0.4) {$2$};
\node at (3,3.2) {$2$};
%\node at (7.5,6.0) {$2$};
\node at (-1.5,1.3) {$a$};
\node at (5.9,1.3) {$a$};
\node at (-3.0,4.0) {$b$};
\node at (1.5, 4.0) {$b$};
\node at (-1.5, 5.8) {$1$};
\draw[gray, thick] (0,0)--(-1.5,2.6);
%\node at (7,3) {\hbox{\Huge $=$}};
\node at (-2.7,1) {\hbox{\huge =}};
\end{tikzpicture}
\caption{(a) 
The $L$-shaped translation surface $X_0$ where the horizontal and vertical sides are identified; (b)  and (c)  are  equivalent surfaces $X_1=$
~\usebox{\smlmat}$X_0$
%$\bigl(
%\begin{smallmatrix}
%1&\frac{1}{2} \\ 0 &1 
%\end{smallmatrix} 
%\bigr)X_0$  
and $X_2=$
~\usebox{\ssmlmat}$X_0$
%$\left(\begin{smallmatrix} 1 & -\frac{1}{2} \cr 0 &1\end{smallmatrix} \right)X_0$, respectively, 
 in the orbit which are triangulated by equilateral triangles.
}
\end{figure}
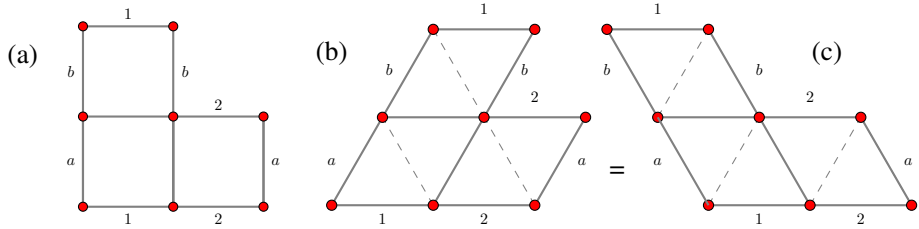

%Observe that:
%\begin{enumerate}
%\item
%the surface in Figure 1 (b) arises from the surface $X_0$  in Figure 1 (a) by shearing by the application of 
%$\left(\begin{smallmatrix} 1 & \frac{1}{2} \\ 0 &1\end{smallmatrix} \right) \in SL(2, \mathbb R)$; and 
%\item
%similarly, the surface in Figure 1 (c) arises from the surface $X_0$   by shearing by the application of 
%$\left(\begin{smallmatrix} 1 & -\frac{1}{2} \\ 0 &1\end{smallmatrix} \right) \in SL(2, \mathbb R)$. 
%\end{enumerate}
%Equivalently, the surface in Figure 1 (c) can be presented as in Figure 1 (d). 

 In Figure 2  we have plotted  an approximation\footnote{This plot was obtained using the method for approximating the entropy of surfaces in the orbit of $X_0$ derived in Section 8.} to the entropy of the surfaces.
 $$\left(\begin{matrix} e^{u} &0\\ 0 & e^{-u}\end{matrix} \right) \left(\begin{matrix} 1 & s \\ 0 &1\end{matrix} \right)X_1
 \hbox{ for } -1/2 \leq s \leq 1/2 \hbox{ and } -\frac{1}{10}  \leq u \leq \frac{1}{10}
 $$
 and indicated the point above $X_1=X_2$.
 We confirm   empirically 
that the entropy is locally minimized at $X_1=X_2$. A simple symmetry argument confirms that $X_0$ is also a critical point.

\begin{figure}[h!]
\centerline{
\includegraphics[height=6.95cm]{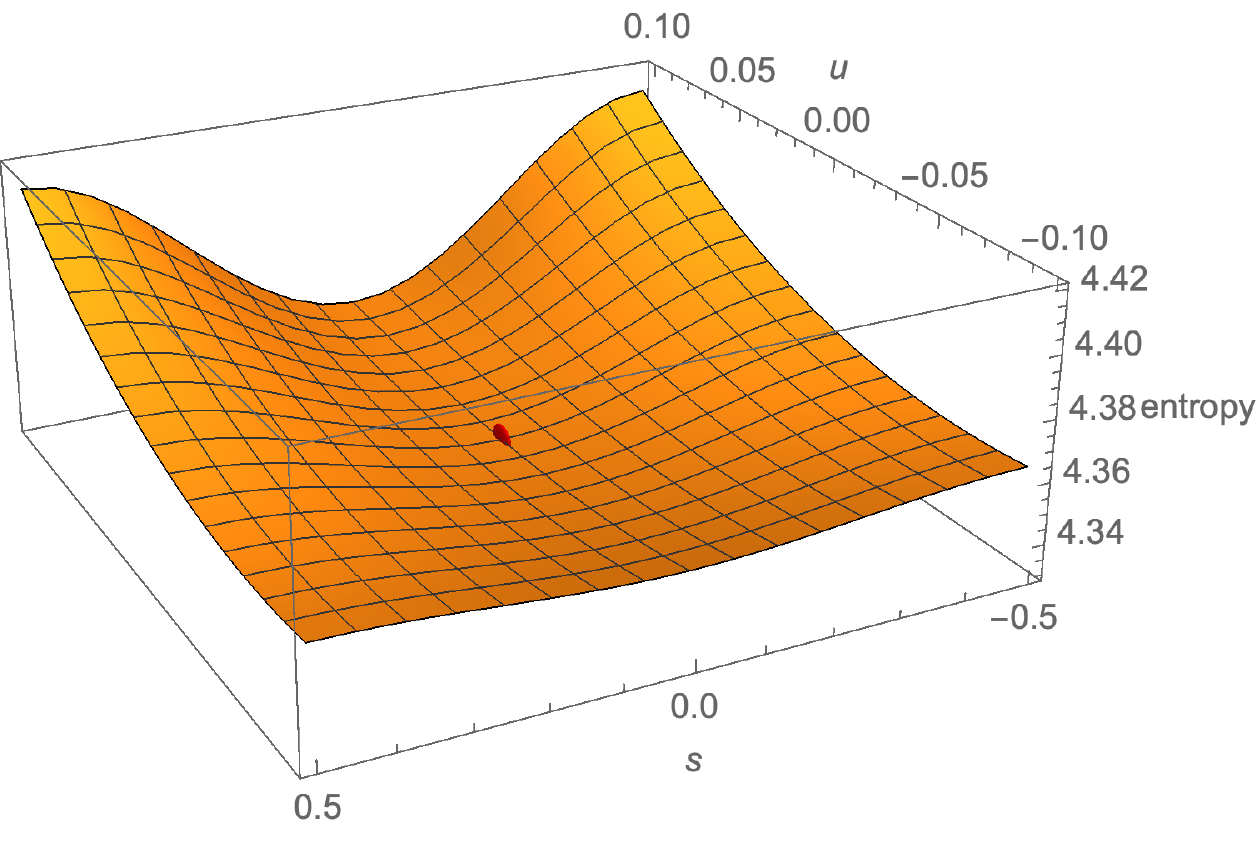}
}
\caption{A numerical approximation to the plot of the entropy of the surface 
%$\left(\begin{smallmatrix} e^{u} &0\\ 0 & e^{-u}\end{smallmatrix} \right) \left(\begin{smallmatrix} 1 & s \\ 0 &1\end{smallmatrix} \right)
~\usebox{\sssmlmat}$X_1$, where $X_1$ is represented by  the surface in Figure 1 (b). The minimum occurring  at $s=0$ and $u=0$ is illustrated.}
\end{figure}
%A simple symmetry argument confirms that   $(u,s) = (0,0)$ is also a critical point.
%\end{example}

If we homothetically scale any  translation surface by a factor $c>0$ then the entropy scales by $1/c$, but the area scales by $c^2$.  Therefore,  it is appropriate to consider translation surfaces  scaled to have unit area, say. 
Let $\mathcal H^1[k, n]$ denote a stratum of unit area surfaces with $n$ singularities, each with the same cone angle 
$2\pi (k+1)$, where $n, k \geq 1$.    
% Let $\mathcal C$ be a connected component of $\mathcal H(k, n)$.
In light of Theorem \ref{main} we conjecture\footnote{This conjecture may be related to the  ``Universal Optimality Conjecture'' (see Conjecture 9.4 in Section 9 of \cite{cohn}).} the following:

\begin{conjecture}
The entropy function $h:\mathcal H^1[k, n] \to \mathbb R^+$ has global minima at equilateral translation surfaces.
\end{conjecture}

In Sections 2 and 3 we present some preliminary results on entropy.  In Section 4 we will present some more examples of surfaces satisfying the hypotheses  in (1).
In Sections 5 and 6 we introduce the main technical  ingredients in the proof:  Montgomery's and 
Bernstein's Theorems, respectively.  In Section 7 we complete the proof of Theorem \ref{main}. In Section 8 we derive a method for approximating the entropy of the types of surfaces we consider in this note.
In the final section we collect together some final comments and questions.

%\begin{rem}[Question]
%In a given strata will there be only finitely many equilateral triangulations?
%Do the number of  minima for the entropy correspond to the different quotients for the Veech group?
%\end{rem}

\section{Translation surfaces  and entropy}

%We can consider translation surfaces represented by identifying opposite sides of polygons in the plane
%with parallel sides.   

Fix a translation surface $X$ with singularity set $\Sigma$.   
A saddle connection $s$ is a straight line on $X$ between two singularities (which does not contain a singularity in its interior).  

%Consider particularly the square tiled translation surfaces for which the vertices of the triangles are all singularities of cone angles
%$2\pi (k+1)$ ($k \geq 1$).

The entropy can be defined in terms of the growth of saddle connection paths, which are geodesics joining singularities. Let $i(s)$ and $t(s)$ denote the initial and terminal singularity, respectively, of an oriented saddle connection $s$.
Given a translation surface $X$, let  $\underline s = s_1 \ldots s_n$ denote an oriented   saddle connection 
path of length  $\ell(\underline s)$ where consecutive oriented saddle connections $s_i$ and $s_{i+1}$ form a locally distance minimizing geodesic.  In particular, the angle between $s_i$ and $s_{i+1}$, for $1 \leq i \leq n-1$, should be greater than or equal to $\pi$ on both sides.
  We write $\ell(\underline s) = \sum_{i=1}^n \ell(s_i)$ where $\ell(s_i)$ denotes the length of $s_i$.  
  %The following is a reformulation of the definition from the introduction
  The following definition is easily seen to be equivalent to the definition from the introduction. 

\begin{definition}\label{entropy}
The entropy  $h(X)$ of  a translation surface $X$ is given by the growth rate of 
saddle connection paths on $X$
$$
h(X) = \lim_{T \to +\infty}
\frac{1}{T} \log  \#\left\{
\underline s \hbox{ : } \ell(\underline s) \leq T
\right\}.
$$
\end{definition}
Whenever $\Sigma \neq \emptyset$ we have that $h(X) > 0$.

There is a useful  alternative formulation which we now present in the next lemma that follows from Definition \ref{entropy} (see \cite{kim}).

\begin{lemma}\label{diverge}
We can write
$$
h(X) = \inf\left\{ t > 0 \hbox{ : } \sum_{\underline s} \exp\left({-t \ell(\underline s)}\right) < +\infty \right\}
$$
where the summation is over all oriented saddle connection paths on $X$.
\end{lemma}
\begin{proof}
Let $h'$ be the infimal value of $t$ for which $\sum_{\underline s }e^{-t\ell(\underline s)}$ converges. Then for $t>h$,
$$
\begin{aligned}
\sum_{\underline s}e^{-t\ell(\underline s)}&=\sum_{n=0}^\infty \sum_{\underline s:n\leq \ell(\underline s)\leq n+1}e^{-t\ell(\underline s)}\cr
&\leq \sum_{n=0}^\infty  \#\left\{
\underline s \hbox{ : } \ell(\underline s) \leq n+1
\right\}e^{-tn}\cr
&=\sum_{n=0}^\infty e^{(h-t+o(1))n}.
\end{aligned}
$$
It follows that $\sum_{\underline s}e^{-t\ell(\underline s)}$ converges when $t>h$, hence $h\geq h'$.\\

Suppose that $h>h'$. Then because the set of $u$ for which 
$$\sum_{\underline s}e^{-u\ell(\underline s)}<\infty
$$
is an interval,
we can choose some $t\in (h',h)$ such that 
$$\sum_{\underline s}e^{-t\ell(\underline s)}<\infty.
$$
Then for $R>0$,

$$e^{tR}\sum_{\underline s}e^{-t\ell(\underline s)}=\sum_{\underline s}e^{t(R-\ell(\underline s))}\geq \sum_{\underline s:\ell(\underline s)\leq R}e^{t(R-\ell(\underline s))}.$$
For $\ell(\underline s)<R$, we have $e^{t(R-\ell(\underline s))}>1$, hence $$e^{tR}\sum_{\underline s}e^{-t\ell(\underline s)}\geq  \#\left\{
\underline s \hbox{ : } \ell(\underline s) \leq R
\right\}.$$ Taking the logarithm of both sides and letting $R$ tend to infinity, we obtain $t\geq h$, which gives a contradiction. Hence $h=h'$.
\end{proof}

\smallskip
We conclude this section by 
introducing a notational device for 
%defining a generalization of a 
certain sequences of saddle connections that will be used in the proof of Lemma \ref{char} in the next section.
\begin{definition}
A \textit{singular} connection, $e$, is a finite sequence of saddle connections, i.e. $e=s_1\ldots s_n$, such that for $1\leq i<n$, $t(s_i)=i(s_{i+1})$ and the angle formed by starting at $s_i$ and moving clockwise\footnote{Translation surfaces have a well define notion of clockwise orientation at every point.} about $t(s_i)$ to $s_{i+1}$ is equal to $\pi$ (Figure \ref{fig:singularconnection}).

\end{definition}

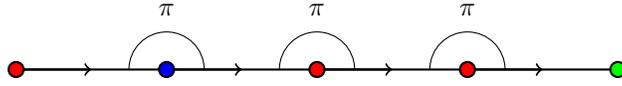
\begin{figure}[h]
    \centering
   \begin{tikzpicture}
   \draw[thick] (0,0)--(8,0);
   \draw[fill=red, thick] (0,0) circle (0.1);
     \draw[fill=blue, thick] (2,0) circle (0.1);
       \draw[fill=red, thick] (4,0) circle (0.1);
         \draw[fill=red, thick] (6,0) circle (0.1);
           \draw[fill=green, thick] (8,0) circle (0.1);
 %   \draw (1.5,0) arc (180:360:0.5);
  %  \draw (3.5,0) arc (180:360:0.5);
   % \draw (5.5,0) arc (180:360:0.5);
       \draw (2.5,0) arc (0:180:0.5);
    \draw (4.5,0) arc (0:180:0.5);
    \draw (6.5,0) arc (0:180:0.5);
    \node[below] at (4,1.0) {$\pi$};
        \node[below] at (2,1.0) {$\pi$};   
        \node[below] at (6,1.0) {$\pi$};
   \draw[thick,->] (0.1,0)--(1,0);
      \draw[thick,->] (2.1,0)--(3,0);
            \draw[thick,->] (4.1,0)--(5,0);
   \draw[thick,->] (6.1,0)--(7,0);

   \end{tikzpicture}
    \caption{Four saddle connections that form a single singular connection.}
    \label{fig:singularconnection}
\end{figure}
Let $\underline s=s_1\dots s_n$ be a saddle connection path. Consecutive saddle connections in $\underline s$ will join at an angle allowed by the condition that $\underline s$ is a geodesic, and so typically it will not be a singular connection.
%Typically, consecutive saddle connections in $\underline s$ may join at any angle allowed by the condition that $\underline s$ is a geodesic (and so the concatenation of these saddle connections will typically not be a singular connection). 
However, exceptionally $\underline s$ may have substrings of saddle connections that form singular connections. Assuming that we only consider singular connections $e$ that are maximal with respect to $\underline s$ (i.e. there does not exist another substring $e'$ of $\underline s$ such that $e$ is a substring of $e'$ and $e'$ is also a singular connection), then $\underline s$ has a unique decomposition of the form $\underline s=e_1\ldots e_m$, where the $e_i$ are singular connections, and for $1\leq i<m$, $e_ie_{i+1}$ is not a singular connection (i.e., the clockwise angle between them is greater than $\pi$).

%{\color{red} Moreover, we assume that $e$ is not a substring of another string with this property. In particular, each $s_i$ may typically be a single saddle connection (when neither the angle with the preceding nor the  succeeding saddle connection, when they exist, is not equal to $\pi$).}

\begin{example}
For the $L$-shaped translation surface we can consider the saddle connection path $s_1s_2s_3s_3s_4$ illustrated in Figure 5.  
Since the clockwise angle between $s_1$ and $s_2$ is $\pi$, and the angle between $s_3$ and itself is $\pi$.
However, the angle between $s_2$ and $s_3$ is $\frac{5}{4}\pi$ and the angle between 
$s_3$ and $s_4$ is $\frac{3}{2}\pi$.  
Thus we  can  denote the singular connections  $e_1=s_1s_2$, $e_2=s_3s_3$ and $e_3=s_4$
and write $s_1s_2s_3s_3 s_4= e_1e_2 e_3$.
\end{example}

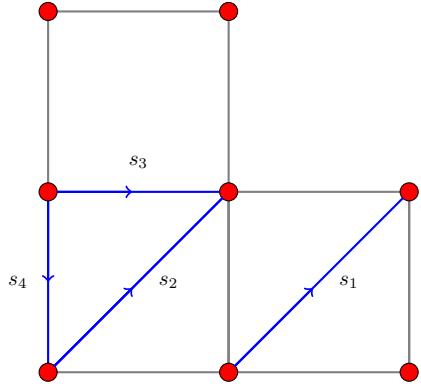
\begin{figure}[h!]
\begin{tikzpicture}[scale=0.8, every node/.style={scale=0.8}]
\draw[gray, thick] (0,0)--(6,0);
\draw[gray, thick] (3,0)--(3,6);
\draw[gray, thick] (0,3)--(0,6);
%\draw[blue, thick,->] (0,3)--(0,4.5);
\draw[gray, thick] (3,0)--(3,3);
\draw[gray, thick] (6,0)--(6,3);
\draw[gray, thick] (3,6)--(0,6);
%\draw[gray, thick] (6,0)--(7,0);
%\draw[gray, thick, dashed] (7,0)--(8,0);
\draw[gray, thick] (3,3)--(6,3);
\draw[blue, thick,->] (0,3)--(0,0);
\draw[blue, thick,->] (0,3)--(0,1.5);
\draw[gray, thick] (3,0)--(3,3);
\draw[gray, thick] (6,0)--(6,3);
\draw[gray, thick] (6,9)--(6,9);
\draw [fill=red] (0,0) circle [radius=0.15];
\draw [fill=red] (0,3) circle [radius=0.15];
\draw [fill=red] (3,3) circle [radius=0.15];
\draw [fill=red] (3,6) circle [radius=0.15];
\draw [fill=red] (0,6) circle [radius=0.15];
%\draw [fill=red] (6,9) circle [radius=0.15];
\draw [fill=red] (3,0) circle [radius=0.15];
\draw [fill=red] (6,0) circle [radius=0.15];
\draw [fill=red] (6,3) circle [radius=0.15];
\draw[blue, thick,->] (3.1,0.1)--(4.4,1.4);
\draw[blue, thick] (3.1,0.1)--(5.9,2.9);
\node at (5,1.5) {$s_1$};
\draw[blue, thick,->] (0.1,0.1)--(1.4,1.4);
\draw[blue, thick] (0.1,0.1)--(2.9,2.9);
\node at (2,1.5) {$s_2$};
\draw[blue, thick] (0.15,3.0)--(2.85,3);
\draw[blue, thick,->] (0.15,3)--(1.4,3);
\draw[blue, thick] (0.1,0.1)--(2.9,2.9);
\node at (1.5,3.5) {$s_3$};
\node at (-0.5,1.5) {$s_4$};
%\node at (1.5,-0.4) {$1$};
%\node at (4.5,-0.4) {$2$};
%\node at (1.5,6.4) {$1$};
%\node at (4.5,3.4) {$2$};
%\node at (-0.4,1.5) {$a$};
%\node at (6.4,1.5) {$a$};
%\node at (-0.4,4.5) {$b$};
%\node at (3.4, 4.5) {$b$};
%\draw[blue, thick] (3,0)--(6,3);
%\draw[blue, thick,->] (3,0)--(4.5,1.5);
%\node at (5,1.5) {$s_1$};
%\draw[blue, thick] (0,3)--(3,6);
%\draw[blue, thick,->] (0,3)--(1.5,4.5);
%\node at (2,4.5) {$s_2$};
%\draw[blue, thick] (0,0)--(0,3);
%\draw[blue, thick,->] (0,0)--(0,1.5);
%\node at (0.5,1.5) {$s_3$};
%\node at (-0.5,4.5) {$s_4$};
\end{tikzpicture}
%\hskip 0.25cm
\caption{The saddle connection path $s_1s_2s_3s_3s_4=e_1e_2e_3$.
}
\end{figure}

\section{Entropy formula for surfaces in $\mathcal H^1[k,n]$ and their $SL(2,\mathbb R)$-orbits}
Fix $X_0\in\mathcal H^1[k,n]$, a unit area square tiled surface 
with $n$ singularities with cone angle $2\pi (k+1)$
satisfying the conditions in (1) in the introduction. Note that $X_0$ will be tiled by $n(k+1)$ squares.

In order to study surfaces $AX_0$, where $A \in SL(2, \mathbb R)$,  we can denote 
by 
$$\Lambda =  \left\{ \left(\frac{a}{\sqrt{n(k+1)}}, \frac{b}{\sqrt{n(k+1)}}\right)  \hbox{ : }
(a,b) \in \mathbb Z^2\setminus\{(0,0)\} \right\}$$
 the scaled standard lattice (minus the origin) and 
$A(\Lambda) = \{  \left(A\left(\begin{smallmatrix} p \\ q \end{smallmatrix}\right)\right)^T \hbox{ : } (p,q) \in \Lambda\}$ its image under the linear action of $A$.  
We can then associate  functions $f_t: SL(2, \mathbb R) \to \mathbb R^+$ for each $t > 0$ by
$$
f_t(A) = \sum_{\underline v \in A(\Lambda)} \exp\left(-t \|\underline v\|\right)
$$
where $ \|\underline v\|$ is the usual Euclidean length of $\underline v \in \mathbb R^2$.

\smallskip
We will relate the entropy  $h(AX_0)$ of $AX_0$ to $f_t(A)$ by taking advantage of the additional structure of the set of saddle connection paths on $AX_0$.

\begin{lemma}\label{scpaths}
Let $X_0 \in \mathcal H^1[k,n]$ satisfy the hypotheses in (1) .
%Assume that  the singularities of $X_0$ have the same cone angle.
Fix $A \in SL(2, \mathbb R)$.  There is a correspondence between the set of oriented saddle connection paths that admit a unique decomposition into $m$ singular connections, and 
$$
 \begin{aligned}
 &\left(x; j_1, (a_1,b_1);  j_2, (a_2,b_2);    \cdots ; j_m, (a_m,b_m)    \right) \cr
 &\in \Sigma \times 
 \underbrace{\left(
 \mathbb Z/{(k+1)}\mathbb Z \times A(\Lambda)\right)
 \times 
  \left(
 \mathbb Z/k\mathbb Z \times A(\Lambda)\right)
 \times \cdots
 \times 
  \left(
 \mathbb Z/k\mathbb Z \times A(\Lambda)\right)}_{\text{$m$ times}}.
% #\times 
% \cdots
% \times
%  \left(
% \mathbb Z_{k} \times \mathbb Z^2 \setminus \{(0,0)\}\right)
% \cr
 \end{aligned}
 $$
 Moreover, we can write 
$$
\ell(\underline s) = \sum_{i=1}^m \|(a_i,b_i)\|
\hbox{ where 
$\|(a_i,b_i)\| = \sqrt{a_i^2 + b_i^2}$.}
$$
\end{lemma}
\begin{proof}
We will proceed by induction on $m$, the number of singular connections in the saddle connection path. 

We begin by looking at the set of saddle connection paths consisting of one singular connection. First note that the  square tiled surface $AX_0$ (which has area 1 and consists of $n(k+1)$ square tiles) covers the 
%homothetically scaled version 
% of the 
  torus,
  $(1/\sqrt{n(k+1)})A(\mathbb R^2/\mathbb Z^2)$, and around the singularities the covering map can be written in the form $w=z^{k+1}$ using local complex coordinates.
  %(recall that $AX_0$'s singularities each have cone angle $2\pi(k+1)$).
  Hence, each oriented singular connection based at some $x\in \Sigma$ projects onto an oriented closed geodesic on the aforementioned torus. The set of such oriented closed geodesics (up to homotopy) is in correspondence with $A(\Lambda)$ and their lengths are given by the lengths of the corresponding vectors. Each oriented closed geodesic on the torus lifts to $(k+1)$ oriented singular connections based at $x$ (the angle between any of the lifts based at $x$ will be equal to $2\pi$ (mod 1)). Hence the set of oriented saddle connection paths on $X$ that consist of one singular connection is in correspondence with $\Sigma \times \left(
 \mathbb Z/{(k+1)}\mathbb Z \times A(\Lambda)\right)$ and the length of the saddle connection path is the length of the corresponding vector in $A(\Lambda)$.
 
 Suppose we have an oriented saddle connection path $\underline s = e_1\ldots e_n$ consisting of $n$ singular connections, $e_i$. Let $e$ be an oriented singular connection. Then $\underline s e$ is a saddle connection path if, $t(\underline s)=i(e)$, the anticlockwise angle between $e_n$ and $e$ is greater than or equal to $\pi$ (see the beginning of Section 2), and the clockwise angle between them is greater than $\pi$ (see the end of Section 2). We have seen that the set of oriented singular connections that begin at a given singularity corresponds to $\left(\mathbb Z/{(k+1)}\mathbb Z \times A(\Lambda)\right)$, where the oriented singular connections can be grouped in $(k+1)$-tuples (corresponding to vectors in $A(\Lambda)$) such that each pair in the tuple forms an angle of $2\pi$ (mod 1) about the singularity. The two angle conditions together eliminate exactly one singular connection from every $(k+1)$-tuple and so the set of oriented singular connections that form a saddle connection path $\underline s e$ with $\underline s$, is in correspondence with $\left(\mathbb Z/{k}\mathbb Z \times A(\Lambda)\right)$. 
 
By using the induction hypothesis on $m=n$ and the above reasoning, it follows that the correspondence in the statement of the Lemma and expression for $\ell(\underline s)$ hold for $m=n+1$ and so we are done by induction.\end{proof}

The functions $f_t(\cdot)$ have the following useful properties.  

\begin{lemma}\label{char}
Let $X_0 \in \mathcal H^1[k,n]$ satisfy the hypotheses in (1) .
%Assume that  the singularities of $X_0$ have the same cone angle.
Fix $A \in SL(2, \mathbb R)$.  
\begin{enumerate}
\item
The function $f_t(A)$ is well defined and $\frac{\partial f_t(A)}{\partial t}<0$ for all $A\in SL(2,\mathbb R)$ and $t > 0$.
\item
 The entropy $h = h(AX_0)> 0$ is the unique solution to $f_h(A) = \frac{1}{k}$
 \end{enumerate}
\end{lemma}

\begin{proof}
Part 1 follows from the definition of $f_t(A)$.

For part 2, we can use Lemma \ref{diverge} and Lemma \ref{scpaths} to write
$$
\begin{aligned}
\sum_{\underline s} \exp\left({-t \ell(\underline s)}\right)
& = n (k+1) \sum_{\underline v \in A(\Lambda)} \exp\left(-t \ell(\underline v)\right)
\sum_{m=0}^\infty \left(
k \sum_{\underline v \in A(\Lambda)} \exp\left(-t \ell(\underline v)\right)
\right)^m \cr
&=  n(k+1) f_t(A)
\sum_{m=0}^\infty \left( k
f_t(A)
\right)^m\cr
&= n\frac{(k+1) f_t(A)}{1 - kf_t(A)}
\end{aligned}
$$
provided $f_t(A) < \frac{1}{k}$.  Moreover, since $f_t(A)$ is strictly decreasing in $t$,  by part 1 of 
Lemma \ref{char}, then  Lemma \ref{diverge} and the above identity, the proof is completed.
\end{proof}

Lemma \ref{char} gives a particularly useful characterization of the entropy.  As a first application we have the following.

\begin{corollary} 
Let $X_0 \in \mathcal H^1[k,n]$ satisfy the hypotheses in (1).
Then the  entropy function $h: SL(2,\mathbb R) X_0 \to \mathbb R^+$ is real analytic when restricted to the orbit $SL(2,\mathbb R) X_0$.
\end{corollary}

\begin{proof}
In the definition of $f_t(A)$ the dependence of the saddle connections on $A$ is real analytic. 
The function also has an analytic dependence on $t$. 
 By Lemma \ref{char} we see that $h(AX_0)$
 satisfies $f_{h(A X_0)}(AX_0) = \frac{1}{k}$
 and  then applying the Implicit Function Theorem gives the result.
\end{proof}

%\begin{rem}
%If $X$ is a square tiled surface 
%(or a  triangulated surface)
%with singularities at the vertices then $X$ factors over a torus $\mathbb R^2/(\mathbb Z + z \mathbb Z)$ where $Re(z) >0$. 
%In particular, we can describe the saddle connections using the lattice $L = \mathbb R^2/(\mathbb Z + z \mathbb Z)$.
%\end{rem}

\section{Examples}

We are interested in square tiled surfaces where the vertex of  each square is a singularity with a common cone angle.  We note that any square tiled surface consisting of $N$ square tiles can be represented by a pair of permutations $(h,v)\in Sym(N)\times Sym(N)$, where $h,v$ represent the gluings of the horizontal and vertical edges of the squares, respectively (see \cite{Schmitusen}). We recall that for translation surfaces $X \in \mathcal H^1(k_1, \ldots, k_n)$ the genus $g$ satisfies
$2g - 2 = \sum_{i=1}^n k_i$ \cite{wright}.
We can consider a few simple examples.

\begin{example}[$O_k$, $k \geq 2$, $h=(1,2,\ldots 2k)$, $v=(1,2)(3,4)\ldots(2k-1,2k)$, 
 see  Figure \ref{fig:ex1}]
This is a translation surface of genus $k$
with two singularities each with cone angle 
%$4\pi (k-1)$
$2\pi k$
 (see \cite{Schmitusen}, Definition 5.3, p.53). 
\begin{figure}[h!]
\centerline{
\begin{tikzpicture}[scale=0.5, every node/.style={scale=0.7}]
\draw[gray, thick] (0,0)--(7,0);
\draw[gray, thick] (10,0)--(17,0);
\draw[gray, thick, dashed] (7,0)--(10,0);
\draw[gray, thick] (0,3)--(7,3);
\draw[gray, thick] (10,3)--(17,3);
\draw[gray, thick, dashed] (7,3)--(10,3);
\draw[gray, thick] (0,0)--(0,3);
\draw[gray, thick] (3,0)--(3,3);
\draw[gray, thick] (6,0)--(6,3);
\draw[gray, thick] (11,0)--(11,3);
\draw[gray, thick] (14,0)--(14,3);
\draw[gray, thick] (17,0)--(17,3);
\draw [fill=black] (0,0) circle [radius=0.15];
\draw [fill=red] (0,3) circle [radius=0.15];
\draw [fill=black] (3,3) circle [radius=0.15];
\draw [fill=red] (3,0) circle [radius=0.15];
\draw [fill=black] (6,0) circle [radius=0.15];
\draw [fill=red] (6,3) circle [radius=0.15];
\draw [fill=black] (17,0) circle [radius=0.15];
\draw [fill=red] (17,3) circle [radius=0.15];
\draw [fill=black] (14,3) circle [radius=0.15];
\draw [fill=red] (14,0) circle [radius=0.15];
\draw [fill=black] (11,0) circle [radius=0.15];
\draw [fill=red] (11,3) circle [radius=0.15];
\node at (1.5,-0.4) {$1$};
\node at (4.5,-0.4) {$2$};
\node at (12.5,-0.4) {$2k-1$};
\node at (15.5,-0.4) {$2k$};
\node at (1.5,3.4) {$2$};
\node at (4.5,3.4) {$1$};
\node at (12.5,3.4) {$2k$};
\node at (15.5,3.4) {$2k-1$};
\node at (-0.4,1.5) {$a$};
\node at (17.4,1.5) {$a$};
%\node[below]
\end{tikzpicture}
}
\caption{A translation surface of genus $k$ and two singularities.}
\label{fig:ex1}

\end{figure}
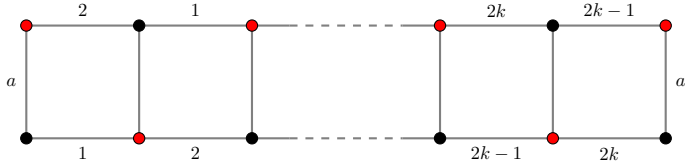
\end{example}

We can next consider two different types of stair examples.

\begin{example}[$St_k = E_{2k-1}$, $k \geq 2$, $h=(1,2)\ldots(2k-3, 2k-2)$, $v=(2,3)\ldots(2k-2,2k-1)$]
This is a translation surface of genus $k$
with one singularity with cone angle 
$2\pi (2k-1)$
(see \cite{Schmitusen}, Definition 5.10, p.61). 

The special case $k=2$ corresponds to Example 1.2.
\end{example}

\begin{example}[$G_k = E_{2k}$, $k \geq 2$, $h=(1,2)\ldots(2k-1, 2k)$, $v=(2,3)\ldots (2k-2, 2k-1)$]
This is a translation surface of genus $k$
with two  singularities each with cone angle 
%$2\pi (2k-3)$
$2\pi k$
 (see \cite{Schmitusen}, Definition 5.8, p.59).

\end{example}

Finally, we can consider a well known example of Forni \cite{forni} and Herrlich-Schmith\"usen \cite{HS}.

\begin{example}[Eierlegende Wollmilchsau\footnote{This literally translates as ``egg-laying wool-milk-pig'' and is  a reference to the many different useful properties this example has.}, $h = (1,2,3,4)(5,6,7,8),$ $v = (1,6)(2,5)(3,8)(4,7)$, see  Figure \ref{fig:ex2}]
This is a translation surface of genus $3$
with four   singularities each with cone angle $4\pi $.

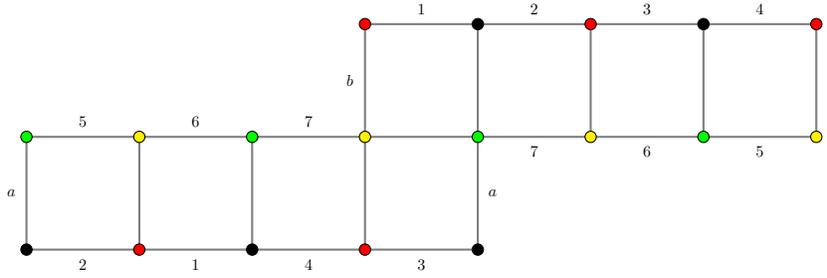
\begin{figure}[h!]
\begin{tikzpicture}[scale=0.5, every node/.style={scale=0.6}]
\draw[gray, thick] (0,0)--(12,0);
\draw[gray, thick] (0,3)--(21,3);
\draw[gray, thick] (9,6)--(21,6);
\draw[gray, thick] (0,0)--(0,3);
\draw[gray, thick] (3,0)--(3,3);
\draw[gray, thick] (6,0)--(6,3);
\draw[gray, thick] (9,0)--(9,3);
\draw[gray, thick] (12,0)--(12,3);
\draw[gray, thick] (15,3)--(15,6);
\draw[gray, thick] (18,3)--(18,6);
\draw[gray, thick] (21,3)--(21,6);
\draw[gray, thick] (12,3)--(12,6);
\draw[gray, thick] (9,3)--(9,6);
\draw [fill=black] (0,0) circle [radius=0.15];
\draw [fill=red] (3,0) circle [radius=0.15];
\draw [fill=black] (6,0) circle [radius=0.15];
\draw [fill=red] (9,0) circle [radius=0.15];
\draw [fill=black] (12,0) circle [radius=0.15];
\draw [fill=red] (9,6) circle [radius=0.15];
\draw [fill=black] (12,6) circle [radius=0.15];
\draw [fill=red] (15,6) circle [radius=0.15];
\draw [fill=black] (18,6) circle [radius=0.15];
\draw [fill=red] (21,6) circle [radius=0.15];
\draw [fill=green] (0,3) circle [radius=0.15];
\draw [fill=green] (6,3) circle [radius=0.15];
\draw [fill=green] (12,3) circle [radius=0.15];
\draw [fill=green] (18,3) circle [radius=0.15];
\draw [fill=yellow] (3,3) circle [radius=0.15];
\draw [fill=yellow] (9,3) circle [radius=0.15];
\draw [fill=yellow] (15,3) circle [radius=0.15];
\draw [fill=yellow] (21,3) circle [radius=0.15];
\node at (1.5,-0.4) {$2$};
\node at (4.5,-0.4) {$1$};
\node at (7.5,-0.4) {$4$};
\node at (10.5,-0.4) {$3$};
\node at (1.5,3.4) {$5$};
\node at (4.5,3.4) {$6$};
\node at (7.5,3.4) {$7$};
\node at (13.5,2.6) {$7$};
\node at (16.5,2.6) {$6$};
\node at (19.5,2.6) {$5$};
\node at (10.5,6.4) {$1$};
\node at (13.5,6.4) {$2$};
\node at (16.5,6.4) {$3$};
\node at (19.5,6.4) {$4$};
\node at (-0.4,1.5) {$a$};
\node at (12.4,1.5) {$a$};
\node at (8.6,4.5) {$b$};
\node at (21.4,4.5) {$b$};
\end{tikzpicture}
\caption{A translation  surface of genus $3$ and $4$ singularities.}
\label{fig:ex2}
\end{figure}
\end{example}

\section{Montgomery's theorem}

In order to analyze $f_t(\cdot)$, and thus use Lemma \ref{char} to study the entropy, 
it is convenient to first study a related function $F_t(L)$, where $L$ is a unimodular  lattice, i.e., 
of the form $A(\Lambda)$ for some $A \in SL(2, \mathbb R)$.
Throughout we consider the lattices  up to rotation.
In particular, this will allow us to use a result of Montgomery.

%Montgomery's theorem describes the dependence of the following expression on $L$

\begin{definition}
We can associate to a unimodular  lattice $L$
and $t > 0$ the function
$$
F_t(L) = \sum_{p \in L} \exp\left(- t \|p\|^2\right)
$$
where
% $L = L\setminus\{0\}$ and 
$\|p\|$ denotes the Euclidean norm.
\end{definition}

We see that $F_t(L) $ is finite provided $t>0$.
Moreover, on this domain 
the function 
$F_t(L)$
has a smooth dependence on $t$ and $L$. 
The next  result   describes lattices which minimize   $F_t(L)$
 \cite{montgomary} (see also \cite{betermin-thesis}, Appendix A).
 Let $L_\Delta$ denote the  equilateral triangular lattice.

\begin{proposition}[Montgomery's Theorem]\label{mont}
For each $t > 0$ and all (unimodular) lattices  $L$ we have that 
$F_t(L) \geq F_t(L_\Delta)$,  with equality iff $L=L_\Delta$.
\end{proposition}

\begin{remark}[Comment on the proof of Proposition \ref{mont}]
There is a standard correspondence between unimodular  lattices $L$ in $\mathbb R^2$  and the standard Modular domain, i.e., $z=x+iy$ with $ - \frac{1}{2} \leq x <  \frac{1}{2}$  
and $|z|\geq 1$, with suitable identifications on the boundary.  Let us denote $L = L_z$.
Let $L_\Delta = \mathbb Z + z \mathbb Z$ be the equilateral  triangular lattice with $z =
% \left(
\frac{1}{2}+ i  \frac{\sqrt{3}}{2} $.
%\right)$.  
%This can be reformulated in terms of positive definite quadratic forms.  In particular,  we  can write
 %$$\theta_f(t) = \sum_{n,m} \exp\left(-2\pi t f(n,m)\right)$$ where $f(x,y) = ax^2 + bxy + cy^2$ and $b^2 - 4ac =-1$.  
% The equilateral lattice corresponds to the quadratic form 
% $g(x,y) = \frac{1}{\sqrt{3}}(x^2 + xy +y^2)$.  
%We want to know that
%$\theta_f(t) \geq \theta_g(t)$, for all $t > 0$, with equality if $f$ and $g$ are related by an element of $SL(2, \mathbb Z)$.
%Furthermore, factorizing  $f(m,n) = a(m + zn)(m + \overline z n)$ we can  use the $SL(2, \mathbb Z)$ action to assume 
% that $z=x+iy$ lies in the modular domain and also write $f(n,m) = f(z)(n,m)$.   
The work of Montgomery established, in particular,  the following properties:
\begin{enumerate}
\item
If
$0 < x < \frac{1}{2}$ and $y \geq \frac{1}{2}$ then $\frac{\partial F_t(L_z)}{\partial x} < 0$ for all $t>0$; and
\item
If
$0 \leq x \leq \frac{1}{2}$ and $x^2 + y^2  \geq 1$  then $\frac{\partial F_t(L_z)}{\partial y} \geq 0$ for all $t>0$
with equality iff $(x,y) = (0,1)$ or $(\frac{1}{2}, \frac{\sqrt{3}}{2})$, i.e., the ramification points on the modular surface.
\end{enumerate}
By  the definitions we have $F_t(L_{x+iy}) = F_t(L_{-x+iy})$ and so we can assume without loss of generality 
that $0 \leq x \leq \frac{1}{2}$.   
  Thus   given a lattice $L_z$ we consider a  path  consisting of a straight line path from $z= x+iy$ to $\frac{1}{2}+iy$ and then 
  a straight line path from  $\frac{1}{2}+iy$ 
  to  $\frac{1}{2}+i\frac{\sqrt{3}}{2}$ along which $F_t(L_{x+iy})$ decreases for any $t>0$ (by (1) and (2), respectively).
%If $x\in \{0, \pm \frac{1}{2}\}$ then by (2)  we can decrease 
%$x$ (and thus $\theta_z(\alpha)$) until one arrives at $(0,1)$ or $\left(\pm \frac{1}{2}, \frac{\sqrt{3}}{2}\right)$.  If $0 <  \pm x < \frac{1}{2} $ then we can reduce the modulus $|x|$ until $x=0$ 
%which  decreases $\theta_z(\alpha)$ by Part 1, then one can apply Part 2 to see that decreasing $y$ to arrive at $(0,1)$ further reduces $\theta_z(\cdot)$
\end{remark}

%\begin{remark}
%We can write $\theta_f(z) = \Theta(i t/2\pi)$ where $\Theta (\cdot)$ is the classical Jacobi Theta function.  
%In particular,  properties include that $\theta_f(1/\alpha) = \alpha \theta_f(\alpha)$.
%\end{remark}

\section{Bernstein's theorem}

To proceed we need  to relate $F_t(\cdot )$ and $f_t(\cdot)$.  This requires a result of Bernstein on completely monotone functions.

\begin{definition}
We call
a smooth function $\psi: \mathbb R^+ \to \mathbb R^+$  \emph{ completely monotone}
if for all $x > 0$:
$$
(-1)^n \frac{d^n\psi}{dx^n} (x) < 0.
$$
\end{definition}

We will be particularly interested in the following example.

\begin{example} The function $\psi(x) = \exp\left(-\sqrt{x}\right)$ is completely monotone (see the corollary on p. 391 of \cite{ms}).  
More generally, given  $\psi_1(x)$ and $\psi_2(x)$ with $\psi_1$ and $\psi_2'$ completely monotone  one has that the composition $\psi_1 \circ \psi_2$ is completely monotone (see Theorem 1 in \cite{ms}).
  We can apply this result with $\psi_1(x)  = \exp\left(-x\right)$ and $\psi_2(x) = x^{\frac{1}{2}}$.
\end{example}

The interest in completely monotone functions is that they are the Laplace transform of positive functions, as is shown in the following classical theorem \cite{bern}.

\begin{proposition}[Bernstein's Theorem]\label{bernstein}
If $\psi$ is completely monotone, then there exists a finite positive Borel measure  $\mu$ on $[0, +\infty)$
such that
$$
\psi(r) = \int_0^\infty \exp\left(-ru\right) d\mu(u).$$
\end{proposition}

An account  appears, for example, in  the book of Widder (see Chapter IV, \S12 \cite{widder}).

\section{Proof of Theorem \ref{main}}

%\section{The B\'etermin's Theorem}

We want to use use Proposition \ref{bernstein} to convert Proposition \ref{mont} for $F_t(\cdot )$ into the  corresponding result for $f_t(\cdot)$, in Proposition \ref{betterman} below.

\begin{proposition}[B\'etermin]\label{betterman}
%The function 
%$P_1(L, t)$
%has a smooth dependence on $t$ and $v$.    Moreover, 
For each $t > 0$  and all lattices $L$   we have that     $f_t(L) \geq f_t(L_\Delta)$ with equality iff $L=L_\Delta$.
\end{proposition}

For completeness we recall the elegant short proof of B\'etermin.

\begin{proof}
First  one uses Proposition \ref{bernstein} to write
$$
 \exp\left(- t \|p\|\right) = \int_0^\infty \exp\left({-u t  \|p\|^2}\right) \rho(u) du
$$
for $t>0$ and $p \in L$.
Summing gives 
$$
f_t(L) - f_t(L_\Delta)= \int_0^\infty \left( 
F_{ut}(L) - F_{ut}(L_\Delta)
\right)\rho(u) du \geq 0
$$
with   $F_{ut}(L) - F_{ut}(L_\Delta) \geq 0$ and with equality iff $L=L_\Delta$,
using Proposition \ref{mont}.
\end{proof}

We can now complete the proof of Theorem \ref{main}.

Let $A \in SL(2, \mathbb R)$ be chosen so that $A X_0$ has a triangulation by equilateral triangles and let 
$B \in SL(2, \mathbb R)$ be an element in the group which does {\it not} correspond to triangulation by equilateral triangles.

We can use the functions $f_t(A)$ and $f_t(B)$ to compare the entropies  $h(A)$ and $h(B)$ of $AX_0$ and $BX_0$, respectively.
By Proposition  \ref{betterman} we know that  that $f_t(A) <  f_t(B)$ for all $t > 0$.
By part 2 of Lemma \ref{char} the entropy $h(A)$ for the surface $AX_0$ is the unique value such that 
$f_{h(A)}(A) = \frac{1}{k}$.  However, by part 1 of Lemma \ref{char} the function $t \mapsto f_t(B)$ is monotone decreasing so the solution $f_{h(B)}(B) = \frac{1}{k}$ implies that $h(B) > h(A)$.

%This is a straightforward argument that involves relating $\sum_{\underline s} e^{-t\ell(\underline s)}$ and 
%$\sum_{n=0}^\infty P_1(L,t)^n$.  

\section{Approximating $h(X_0)$}

Let $X_0$ denote a square tiled surface in $\mathcal H^1[k,n]$. 
%We develop
 In this section we present
 a method for finding arbitrarily good approximations to $h(A)$, for a given $A\in SL(2,\mathbb R)$, and calculate the error terms of these approximations. We will then use this method to approximate $h(\Delta)$, the entropy of the equilateral surface in Example \ref{lshape}.\\

Let $N\in\mathbb N$ and define $\mathbb Z_N=\{n\in\mathbb Z:|n|\leq N\}$. We define the finite square lattice 
$$\Lambda_N =  \left\{ \left(\frac{a}{\sqrt{n(k+1)}}, \frac{b}{\sqrt{n(k+1)}}\right)  \hbox{ : }
(a,b) \in \mathbb Z_N^2\setminus\{(0,0)\} \right\}.$$

\noindent
Fix $t>0$. We can then define an $N^{th}$ approximation to $f_t(A)$ by considering a truncation of the infinite series in the definition of $f_t(A)$:

$$f_t^{(N)}(A)=\sum_{\underline v\in A(\Lambda_N)}\exp({-t\ell(\underline v)}).$$
Note that the first few derivatives of $f_t^{(N)}(A)$ also give approximations to the respective derivatives of $f_t(A)$.
Finally,  we note that the region 
$$P(A) = \left\{A\left(x,y\right) \hbox{ : }   \frac{1}{2\sqrt{n(k+1)}} \leq x, y \leq \frac{1}{2\sqrt{n(k+1)}} \right\}$$
has area $\hbox{\rm Area}(P(A))=  \frac{1}{n(k+1)}$ and  its translates by $A \left(  \frac{1}{\sqrt{n(k+1)}} \mathbb Z^2\right)$ tile
the plane  $\mathbb R^2$.

The following simple lemma allows us to bound the error in the approximations.

\begin{lemma}\label{lem:inequalities}
Fix $A\in SL(2,\mathbb R)$. We define 
%$Q=\{(m,n)\in \mathbb Z^2: m\neq 0\hbox{ and }  n\neq 0\}$ 
 $$d(A):=\inf_{\underline x\in\mathbb R^2:||\underline x||_2=1}||A(\underline x)||_2
 \hbox{ and } D(A):= \hbox{\rm diam} (P(A)).$$ 
 Let $g:\mathbb R^+\rightarrow \mathbb R^+$ be the   function
 % such that $\int_0^\infty g(x)dx <\infty$. 
 $g(R) = \exp(-t R)$.
 Then the following inequalities hold:
\begin{enumerate}
    \item $$\sum_{\underline v \in A(\Lambda)}g(\ell(\underline v))\leq 2\pi  n(k+1)\exp(t D(A)) \int_0^\infty Rg(R)dR;  \hbox{ and}$$
    
    \item For $N\in\mathbb N$,
    $$
    \sum_{\underline v \in A(\Lambda)}g(\ell(\underline v))-\sum_{\underline v \in A(\Lambda_N)}g(\ell(\underline v))
    \leq 2\pi   n(k+1) \exp(t D(A)) \int_{d(A) N/ \sqrt{n(k+1})}^\infty Rg(R)dR.
    $$
  %   $$\sum_{\underline x\in g(\Lambda'):||\underline x||_2\geq R_0}f(||\underline x||_2)\leq 2\pi \int_{R_0-d}^\infty Rf(R)dR.$$
\end{enumerate}
\end{lemma}

\begin{proof}
The proof follows easily by bounding each term 
$$
g(\ell(\underline v)) \leq \frac{\exp\left(t D(A)\right)}{\hbox{\rm Area}(P(A))} \int_{\underline v + P(A)} g(\sqrt{x^2+y^2})dxdy.
$$
%and 
%then bounding   the series by  integrals of $g\left(\sqrt{x^2+y^2}\right)$ over suitable domains.
 For part (1) we integrate over  $\mathbb R^2$ and for part (2) over $\left\{(x,y) \hbox{ : } x^2 + y^2 \geq  
 \frac{d(A)^2N^2}{n(k+1)}\right\}$, in both cases using polar coordinates.
\end{proof}

We will now use the $f_t^N(A)$ and the above lemma to approximate $h(\Delta)$, where $h(\Delta)$ denotes the entropy of the equilateral surface in Example \ref{lshape} (the unique equilateral surface in $\mathcal H^1[2,1]$.\\

It follows from Lemma \ref{char} that $h(\Delta)$ is the unique $t>0$ such that $f_t(\Delta)=1/2$.\\

By applying inequality (2) from Lemma \ref{lem:inequalities} to $f_t(\Delta)-f_t^{(N)}(\Delta)$, we obtain an upper bound for $f_t(\Delta)-f_t^{(N)}(\Delta)$ which we denote by $E_t^{(N)}(\Delta)$. Next observe the following inequalities:

$$f_t^{(N)}(\Delta)\leq f_t(\Delta)\leq f_t^{(N)}(\Delta)+E_t^{(N)}(\Delta),$$

\noindent
where each of the terms are decreasing in $t$. \\

Let $h_L^{(N)}(\Delta)$ denote the unique $t>0$ such that $f_t^{(N)}(\Delta)=1/2$ and let $h_U^{(N)}(\Delta)$ denote the unique $t>0$ such that $f_t^{(N)}(\Delta)+E_t^{(N)}(\Delta)=1/2$.\\

It follows from the previous inequality that for all $N\in\mathbb N$, $h_L^{(N)}(\Delta)\leq h(\Delta)\leq h_U^{(N)}(\Delta)$. Because $f_t^N(\Delta)$ converges to $f_t(\Delta)$ and $E_t^{(N)}(\Delta)$ converges to 0 as $N\rightarrow \infty$, we obtain arbitrarily close bounds to $h(\Delta)$ by taking $N$ sufficiently large.\\

We will first calculate $E_t^{(N)}(\Delta)$ using Lemma \ref{lem:inequalities} and then compute the bounds for $N$ sufficiently large. Note that $d(\Delta)$ is  the smallest  singular values of $\Delta$,  i.e. the square roots of the smallest eigenvalue of $\Delta^*\Delta$, where $\Delta^*$ denotes the adjoint of $\Delta$. By a standard calculation, one can show that $d(\Delta) = 0.759836 \ldots$. 
In the present setting the diameter estimate can be taken  to be $D(\Delta) = 1.07457\ldots$. 
Then we can apply Lemma \ref{lem:inequalities} to get
$$
\begin{aligned}
E_t^{(N)}(\Delta)&= 2\pi \exp(t D(\Delta)) \int_{d(\Delta) N/ \sqrt{n(k+1)}}^\infty R \exp(-tR) dR \cr
&=
2\pi \frac{\exp(t D(\Delta))}{t^2} \exp\left(- \frac{t d(\Delta)N}{\sqrt{n(k+1)}}\right) \left(\frac{td(\Delta)N}{\sqrt{n(k+1)}} + 1\right).
% \frac{\pi}{2}\exp \left(D(\Delta)\left( t 
 % \exp(- D(\Delta) N/ \sqrt{n(k+1)} t ) ( D(\Delta) N \sqrt{n(k+1)} t + 1)
%- 2N/  \sqrt{n(k+1)}\right)\right)\left(1 + 2D(\Delta)N/ \sqrt{n(k+1)} \right).
% \cr
%E_t^{(N)}(\Delta)&=2\pi \exp(e(A)) \int_{2d(\Delta)N}^\infty R g(R) dR\cr
%&=\frac{2\pi}{t^2}
%\exp(e(A)-2d(\Delta)Nt) (2d(\Delta)Nt+1).
\end{aligned}$$\\

Using Mathematica's NSolve with working precision equal to 30, we solve $f_t^{(N)}(\Delta)=1/2$ for $t$, with $N=100$ to obtain 
$$h_L(\Delta)=4.34934504614150290303138902137\ldots$$
Again, using NSolve, we numerically solve $f_t^{(N)}(\Delta)+E_t^{(N)}(\Delta)=1/2$ for $t$, using the expression for $E_t^{(N)}(\Delta)$ with $N=100$ to also get
$$h_U(\Delta)=4.34934504614150290303138902137 \ldots$$
Hence we see that $h(\Delta)=4.34934504614150290303138902137 \ldots$ (up to 29 decimal places).

\begin{remark}
We can use the same method as in the proof of Lemma \ref{lem:inequalities} to deduce other properties of $f_t$. For instance, by setting $N=100$ and approximating the partial derivatives of $f_t$ we can show that the Hessian of $f_t(A)$ at the equilateral surface with $t=4.349\ldots$, is non-degenerate by showing that the determinant of the Hessian is approximately equal to $0.0825337\ldots>0$.
\end{remark}

\section{Final comments and questions}

\begin{enumerate}
\item
We can also consider the entropy function on general strata, without the additional restriction that the singularities share the 
same cone angle.  In this broader context it is not clear what the correct candidates for the global minima are, let alone how to prove they minimize the entropy. We suspect that they won't be equilateral surfaces by analogy with the minimization problem for finite metric graphs where metrics that minimize entropy have edge lengths proportional to the logarithm of the product of the valencies of the edge's vertices (see \cite{lim}).

\item It is natural to ask if the entropy function is smooth, by the analogy with Riemannian manifolds with negative sectional curvature \cite{kkpw}.
%It would also  be interesting to relate, if possible,  the surfaces realizing the global minima to the Veech group.

\end{enumerate}

\medskip
 We are grateful to the referee for suggesting the following questions.

\medskip
\begin{enumerate}

\item[(3)]
Let $X_0$ be a square-tiled surface satisfying the hypotheses in (1). 
Is  the  number of equilateral surfaces in $SL(2,\mathbb R)X_0$  related 
 %to  %Then the number of square-tiled surfaces (with singularities at the vertices of the squares ) in $SL(2,\mathbb R)X$
to the Veech group of $X_0$ in $SL(2,\mathbb Z)$?
% which follows from the orbit stabilizer theorem.
%Let $A$ be the matrix that turns a right angled triangle with vertices $\{(0,0),(1,0),(0,1)\}$ into an equilateral triangle with horizontal edge and apex above the horizontal edge. Then if $Y$ is a square-tiled surface, $A(Y)$ is an equilateral surface. Furthermore, the set of equilateral surfaces in $SL(2,\mathbb R)X$ is in one-to-one correspondence with the set of square-tiled surfaces in $X$ (this follows from the fact that $A$ is invertible).
%Hence, the number of equilateral surfaces in $SL(2, \mathbb R)X$ is given by the index of the Veech group of $X$ in $SL(2,\mathbb Z)$.
The Veech groups for our examples in \S 3 are computed in 
\cite{Schmitusen}. This question could potentially be studied by looking at the $SL(2,\mathbb Z)$-orbit of $X_0$ using SageMath (see \url{http://www.sagemath.org}).
%\footnote{We thank the referee for this suggestion.}

\item[(4)] Since the entropy on the $SL(2, \mathbb R)$-orbit  does not depend on rotating the square tiled surface, it descends to a function on $\mathbb H/\Gamma$, where $\Gamma$ is the Veech group of the surface. A natural question is how this function behaves on this quotient, for example, how does it behave in the cusp?
%\footnote{This question was suggested by the referee.}

\end{enumerate}

\end{document}